\documentclass[]{amsart}
\usepackage{amssymb,amsmath,epsfig,graphics,latexsym,psfrag}
\usepackage[english]{babel}
\usepackage[all]{xy}
\usepackage{amscd, amssymb, amsmath, amsthm, graphics}
\usepackage{amsmath,amsfonts,amsthm,amssymb}
\usepackage{latexsym,amsmath}
\usepackage{graphicx,psfrag}
\usepackage{mathrsfs,stmaryrd}
\usepackage[final]{pdfpages}

\usepackage{hyperref}

\newtheorem{theorem}{Theorem}[section]
\newtheorem{corollary}[theorem]{Corollary}
\newtheorem{lemma}[theorem]{Lemma}

\newtheorem{proposition}[theorem]{Proposition}

\newtheorem{remark}[theorem]{Remark}

\theoremstyle{definition}

\newtheorem{convention}[theorem]{Convention}
\newtheorem{notation}[theorem]{Notation}

\begin{document}

\title[On the fiberedness of surgery 3-manifolds]
{On the fiberedness of surgery 3-manifolds}

\author{Yi Ni}
\address{Department of Mathematics, 
California Institute of Technology, Pasadena, CA 91125, USA}
\email{yini@caltech.edu}

\author{Zhongzi Wang}
\address{School of Mathematical Sciences, Peking University, Beijing 100871, People's Republic of China}
\email{wangzz22@stu.pku.edu.cn}

\date{\today}

\begin{abstract}
Let $M$ be a closed orientable 3-manifold  and $k$ be a knot in $M$.   Then the Dehn surgery of $M$ along $k$ with slope $\alpha$ is not surface fibered  for all but a sparse set of slopes. 
\end{abstract}

\maketitle

\tableofcontents


\section{Introduction}

Both Dehn surgery  (filling)
of 3-manifolds and (surface) fibration of 3-manifolds are topics of primary interests in 3-manifold topology.
Let us recall those notions of 3-manifold theory. We assume that all $3$-manifolds are orientable.

A compact 3-manifold $M$ is {\it fibered}, if $M$ is fibered by a compact surface $F$, that is, $M$ is an $F$-bundle over the circle $S^1$. We can also view $M$ as the mapping torus of a homeomorphism $\phi:F\to F$, denoted by $\mathcal M(F,\phi)$.
 If a compact 3-manifold $M$ is fibered by a compact surface $F$, then $M$ is prime. Furthermore, if $F$ is not a $2$-sphere, then $M$ is irreducible and the fiber $F$ is a proper incompressible surface.


Suppose that $M$ is a 3-manifold with $T$ being a torus component  of $\partial M$, and $\alpha$ is a slope in $T$. 
The  3-manifold $M(\alpha)$ is obtained 
from $M$ by 
capping off $T$ with a solid torus $S^1\times D^2$ so that its meridian 
is glued to $\alpha$. We have $$M(\alpha)=M\cup_{\alpha=\partial D^2}S^1\times D^2.$$ We call $M(\alpha)$ the {\it Dehn filling} of $M$ on $T$ with slope $\alpha$.

Let $X$ be a $3$-manifold and $k$ be a knot in $X$, $N(k)$ be a tubular neighborhood of $k$ in $X$,  $T=\partial N(k)$, and $\alpha$ be a slope on $T$.
Then \[X(k, \alpha)=(X\setminus N(k))\cup_{\alpha=\partial D^2}S^1\times D^2,\] the Dehn filling of $X\setminus N(k)$ on $T$ with slope $\alpha$, is called the {\it Dehn surgery} of $X$ along $k$ with slope 
$\alpha$.

Now assume that $M$ is an irreducible $3$-manifold with $T$ being a torus component of $\partial M$, and $\alpha$ is a slope on $T$.
Let $A_T$ be the set of all slopes on $T$. Let \[\mathbb P(H_1(T;\mathbb Q))\cong \mathbb Q\mathbb P^1\cong\mathbb Q \cup \{\infty\}\] be the projectivization of $H_1(T;\mathbb Q)$, then
there is a natural bijection \[\rho: A_T\to \mathbb P(H_1(T;\mathbb Q)),\] 
and hence a natural embedding
$$\rho: A_T\to \mathbb P(H_1(T;\mathbb Q))\subset \mathbb P(H_1(T;\mathbb R))\cong \mathbb R\mathbb P^1.$$

Given two slopes $\alpha,\beta$, let $\Delta(\alpha,\beta)$ be the minimal number of intersection points between a simple closed curve with slope $\alpha$ and a simple closed curve with slope $\beta$, called the {\it distance} between $\alpha$ and $\beta$.

In order to state our main result, we introduce the following concept: A subset $S \subset A_T$ is {\it sparse}, if it has at most finitely many accumulation points in $\mathbb P(H_1(T;\mathbb R))$. 
Let $(P)$ be a property on $3$-manifolds. We say $M(\alpha)$ has property $(P)$ for all but a sparse set of slopes if there is a sparse set $S\subset A_T$ such that for all $\alpha\in A_T\setminus S$, $ M(\alpha)$ has property $(P)$.


\begin{theorem} \label{main1} Suppose that $M$ is an orientable irreducible $3$-manifold with $\partial M$  an  incompressible torus $T$.
Then the set $\{\alpha\in A_T| M(\alpha)\, \text{ is fibered}\}$ is sparse. Equivalently speaking, for all but a sparse set of slopes, $M(\alpha)$ is non-fibered. 
\end{theorem}

 Let $Y$ be the JSJ piece of $M$  containing $T=\partial M$.
Then $Y$ is either  a hyperbolic piece or a Seifert piece, see Theorem \ref{Thurston}.



The  proof of  Theorem \ref{main1} when $Y$ is a hyperbolic piece (Proposition \ref{hyperbolic}) is subtle. The main route is: By applying the theory of Thurston norm \cite{Th2}, 
 Gabai and Lackenby's results about the invariance of Thurston norm after Dehn filling \cite{Ga,La}, and some  observations in 3-manifold topology and convex geometry,
the proof of  Theorem~\ref{main1} when $Y$ is a hyperbolic piece is reduced to the first author's result about  Dehn surgery on knots in the product manifolds
\cite{Ni}. More detailed outline 
of the proof is given after Proposition \ref{prop:FiniteSlopes}.

The proof of  Theorem~\ref{main1} when  $Y$ is  a Seifert piece (Proposition~\ref{Seifert}) is rather direct. 

   
\begin{corollary}\label{surgery}
Let $X$ be a closed orientable 3-manifold  and $k$ be a knot in $X$.   Then $X(k, \alpha)$ is non-fibered for all but a sparse set of slopes. 
\end{corollary}



\begin{remark}
Samuel Taylor informed us that there is a strong restriction on many fibered fillings of a hyperbolic cusped $3$-manifold \cite[Theorem~1.5]{LMST}.  
\end{remark}


\vspace{5pt}\noindent{\bf Acknowledgements.}\quad  The first author was
partially supported by NSF grant number DMS-1811900 and a Xianfeng Scholarship. 
The second author was partially supported by NSFC grant No. 125B2006. We thank Samuel Taylor for informing us \cite[Theorem~1.5]{LMST}.


\section{Preliminaries}

\begin{notation}
For a compact $n$-manifold $M$ and a codimension-1 proper sub-manifold $F$,  we use $M\bbslash F$ to denote the compact 
$n$-manifold obtained by cutting $M$ open along $F$.
\end{notation}

We state Thurston's geometric decomposition theorem of Haken 3-manifolds which builds on
   Jaco--Shalen--Johannson's torus decomposition theorem  (see \cite{Th1},  \cite{Ha}).

\begin{theorem} \label{Thurston}Let $M$ be a compact orientable irreducible $3$--manifold with (possibly empty) boundary consisting of incompressible tori.
Up to isotopy, there is a unique minimal finite collection of disjoint embedded incompressible tori $\mathcal{T}$  in $M$ (JSJ tori of $M$)  such that any component  of $M\bbslash \mathcal{T}$ (JSJ pieces of $M$) either is
a Seifert manifold (Seifert piece), or
 admits a hyperbolic structure with finite volume (hyperbolic piece).
\end{theorem}

We will use the following convention in the next 3 sections.

\begin{convention}\label{conv:M}
Let $M$ be a compact orientable irreducible $3$-manifold with $\partial M$ being an  incompressible torus $T$. Let $\lambda$ be the unique slope on $T$ which is represented by a rationally null-homologous curve in $M$.
Let $\alpha$ be a slope on $T$, $k_{\alpha}\subset M(\alpha)$ be the core of the Dehn filling. Let $Y$ be the JSJ piece containing $T$. 
\end{convention}

\begin{lemma}\label{sparse1}
The union of finitely many sparse sets is sparse. 
 \end{lemma}
 
 \begin{proof}

Let $S_1,...,S_n$ be sparse sets in $A_T$, which can be identified with \[\mathbb P(H_1(T;\mathbb{Q}))\subset\mathbb P(H_1(T;\mathbb{R})).\] By point set topology the set of accumulation points of $S=\cup_{i=1}^n S_i$ is the union of the set of accumulation points of $S_i$. 
 \end{proof}

Given a slope $\beta$ and a positive integer $p$, define
\[S(\beta, p)=\{\alpha\in A_T|\Delta(\alpha,\beta)=p\}.\]
More generally,
given a set $R$ of slopes, let
\[
S(R,p)=\{\alpha\in A_T|\Delta(\alpha,\beta)=p \text{ for some }\beta\in R\}=\bigcup_{\beta\in R}S(\beta, p).
\]

\begin{remark}\label{sparse2} (1) $A_T$ itself is not sparse, since the closure of $A_T$ is $\mathbb P(H_1(T;\mathbb{R}))$.

(2) Any finite set is sparse. For a given integer $p$ and slope $\beta$, the set $S(\beta, p)$ is sparse. More generally,
given a finite set $R$ of slopes, $S(R,p)$ is sparse.
\end{remark}

We use $N(K)$ to denote the twisted $I$-bundle over the Klein bottle $K$.

\begin{lemma}\label{JSJ}
Suppose that $M\ne N(K)$.
Let $\mathcal T$ be the JSJ tori of $M$,  $M\bbslash \mathcal T= \{Y, M_1, ...., M_k\}$.
Then for all but a sparse set $S_0$ of slopes,  $M(\alpha)$ is still an irreducible 3-manifold with JSJ tori $\mathcal T$ and $M(\alpha)\bbslash \mathcal T= \{Y(\alpha), M_1, ...., M_k\}$. Moreover, if $Y$ is a Seifert piece  (or hyperbolic piece, resp.),  then $Y(\alpha)$ 
is still a Seifert piece which extends the Seifert fibration of $Y$ (or hyperbolic piece, resp.) for all slopes not in $S_0$.
This set $S_0$ can be chosen to be finite if $Y$ is hyperbolic.
\end{lemma}

\begin{proof}
We first consider the case $\mathcal T\ne \emptyset$. 

To prove the lemma, we need only to prove for all but a sparse set of slopes, $Y(\alpha)$ is a JSJ piece of $M(\alpha)$.

\noindent{\bf Case 1.} $Y$ is a hyperbolic piece: $Y(\alpha)$ is still hyperbolic for all but finitely many slopes by Thurston's hyperbolic Dehn filling theorem \cite{Th2}, and then every boundary component of $Y(\alpha)$
must be $\pi_1$-injective and  $Y(\alpha)$ must be a hyperbolic piece of $M(\alpha)$.

\noindent{\bf Case 2.} $Y$ is a Seifert piece: Let $O(Y)$ be the base orbifold of the Seifert fibered space $Y$.
Let $f_0\subset T\subset Y$ be the $S^1$-fiber of $Y$.
Since $Y\ne N(K)$, it is known that $\chi(O(Y))<0$ and then the Seifert fibration is unique up to isotopy. Let $p=\Delta(\alpha,f_0)\ne 0$, then $Y(\alpha)$ is still a Seifert manifold whose $S^1$ fibration extends that of $Y$,  and
$$\chi(O(Y(\alpha)))=\chi(O(Y))+1/p.$$
Clearly with finitely many exceptional $p$,  $\chi(O(Y(\alpha)))<0$. For each  fixed integer $p$, 
the set $S(f_0,p)$  is sparse in $A_{T}$ (Remark~\ref{sparse2}). Since the union of finitely many sparse sets is still sparse (Lemma~\ref{sparse1}), we proved that for all but a sparse set of slopes,  $Y(\alpha)$ is a Seifert manifold with $\chi(O(Y(\alpha)))<0$.
So every boundary component of $Y(\alpha)$
must be $\pi_1$-injective.

Now every JSJ piece $M_i$ adjacent to $Y$ is either a hyperbolic piece or a Seifert piece whose
$S^1$-fiber is different from the fiber of $Y$, therefore different from the fiber of $Y(\alpha)$. 
So $Y(\alpha)$ is a Seifert piece of $M(\alpha)$.

When $\mathcal{T}=\emptyset$, $M$ is Seifert or hyperbolic, the argument is similar and more direct.
\end{proof}

\begin{lemma}\label{topology1}
Suppose that $M$ is fibered by a surface $F$. Then each JSJ piece $M_i$ of $M$ is fibered by a surface $F'_i$,
where $F'_i$ is a component of $F\cap M_i$ after some isotopy of $F$.
\end{lemma}

\begin{proof}
Suppose that $M=\mathcal M(F,\phi)$, the mapping torus of $\phi:F\to F$. By Nielsen--Thurston's theory for surface automorphisms \cite{FM},  there is a minimal reducible curve 
system $\Gamma_\phi\subset F$, which is unique up to isotopy, such that $F\bbslash \Gamma_\phi=\cup_{i=1}^k  F_i$, where
 each $ F_i$ is a minimal union of components of $F\bbslash \Gamma_\phi$ such that $\phi(F_i)=F_i$, and each $\phi|F_i$ is either pseudo-Anosov or periodic. Then the JSJ pieces of $M$ are exactly $\mathcal M( F_i,\phi|F_i)$, $1\le i\le k$. Each $\mathcal M( F_i,\phi|F_i)$ is a hyperbolic piece or a Seifert piece if $\phi|F_i$ is pseudo-Anosov or periodic, respectively. Each $\mathcal M( F_i,\phi|F_i)$ can be written as 
 $\mathcal M(F'_i, \phi^{n_i}|F'_i)$, where $F'_i$ is a component of $ F_i$ and $n_i$ is the number of components of $F_i$.
\end{proof}


\begin{proposition}\label{zero}
For each $\alpha \in A_T\setminus\{\lambda\}$,

(1) The map $H_1(T;  \mathbb R)\to H_1(M(\alpha);  \mathbb R)$ induced by the inclusion $T\subset M(\alpha)$ is trivial.

(2) If $M(\alpha)$ is fibered by a surface $F$, then the algebraic intersection number $[F]\cdot [k_\alpha]=0$.
\end{proposition}

 \begin{proof} (1) Let $j: T=\partial M\to  M$ be the inclusion. Then we have 
$$j_*: H_1(T;  \mathbb R)=H_1(\partial M;  \mathbb R)\to H_1( M;  \mathbb R).$$
By the half-die and half-survive Lemma, $\dim\ker j_*=1$. In fact, it is generated by the slope $\lambda$ in Convention~\ref{conv:M}.

Let $\alpha$ be any slope other than $\lambda$.
Let $i: T\to M \to M(\alpha)$ be the inclusion. Then we have 
$$i_*: H_1(T;  \mathbb R)\to H_1( M;  \mathbb R)\to H_1( M(\alpha);  \mathbb R)=H_1( M\cup_{\alpha=\partial D^2} S^1\times D^2;  \mathbb R).$$
Since $j_*(\lambda)=0 \in H_1( M;  \mathbb R)$, we have that $i_*(\lambda)=0   \in H_1( M(\alpha);  \mathbb R)$.
Since $\alpha$ bounds a disk $D^2$ in $S^1\times D^2$, we obtain that $i_*(\alpha)=0\in H_1( M(\alpha);  \mathbb R)$.
Since $\alpha$ and $\lambda$ form a basis of $H_1(T; \mathbb R)$, we conclude that 
$i_*=0$.

(2) Since $k_\alpha$ can be pushed into $T$, by (1) $k_\alpha=0\in H_1( M(\alpha); \mathbb R)$.
If $M(\alpha)$ is fibered by a surface $F$, then the algebraic intersection number $[F]\cdot [k_\alpha]=0$.
\end{proof}




\section{The case of hyperbolic boundary piece}

We state   Theorem \ref{main1} when $Y$ is a hyperbolic  piece as follows.

\begin{proposition} \label{hyperbolic} Suppose that
 $Y$ is hyperbolic. Then the set $\{\alpha\in A_T| M(\alpha)\, \text{ is fibered}\}$ is sparse.
\end{proposition}

We  collect some notions and results which will be used in the proof of Proposition~\ref{hyperbolic}.
  
 For a finite-sided $n$-dimensional convex polytope $D\subset \mathbb R^n$, $\partial D$ has a natural cell complex structure, where each $k$-dimensional cell is the interior of a finite-side $k$-dimensional convex polytope. 
 We call each $(n-1)$-dimensional cell of $\partial D$ an {\it open face} of $D$.
  
Given a compact, oriented surface $F$, let $F_1,\dots,F_n$ be its components, define
 \[
\chi_-(F)=\sum_{i=1}^n\max\{-\chi(F_i),0\}.
 \]

Let $X$ be a compact oriented $3$--manifold with boundary consisting of tori.
The {\it Thurston norm} of $X$ \cite{Th2} is a continuous function $$th_X:  H_2(X,\partial X; \mathbb R)\to \mathbb R_{\ge0}$$
  satisfying  $$th_X(x)=\min\{\chi_-(F)| \text{$F$ is a compact, properly embedded surface representing $x$}\},$$
  for each $x \in H_2(X, \partial X; \mathbb Z)$ and 
\[th_X(\lambda y)=|\lambda|\cdot th_X(y), \text{ for every }\lambda\in\mathbb R, y\in H_2(X, \partial X; \mathbb R).\]

A homology class $x \in H_2(X, \partial X; \mathbb{R})$ is a {\it rational class} if some nonzero multiples of $x$ can be represented by compact surfaces;
  call $x$ a {\it fiber class}, if $x$ can be represented by a compact proper surface $F$ such that $X$ is fibered by $F$; call $x$ a {\it rational fiber class}, if a nonzero rational multiple of $x$ is a fiber class. 
  
We call a compact surface $F$ representing $x$ {\it norm-minimizing}, if $\chi_-(F)=th_X(x)$. We say $F$ is {\it taut} if it is norm-minimizing, no torus component of $F$ is homologous to a $2$--sphere,
the union of any nonempty collection of components of $\partial F$ is not null-homologous in $H_1(\partial X;\mathbb{R})$, and the union of any nonempty collection of components of $F$ is not null-homologous in $H_2(X,\partial X;\mathbb{R})$. It is known that every nontorsion integer homology class is represented by a taut surface.
 
 \begin{theorem}\label{norm}\cite{Th2}
If $X$ is hyperbolic, then the Thurston norm
 $th_X$ is a norm on $H_2(X,\partial X; \mathbb R)$   whose unit ball 
 $B_X\subset \mathbb R^\beta$ is a finite-sided convex polytope of dimension $\beta$,  where $\beta=\dim H_2(X,\partial X; \mathbb R)$.
 Moreover,
 
 (1) All vertices of  $B_X$ are rational classes.
 
 (2) Fiber classes  correspond to exactly all lattice points lying in the cone over the union of some open faces of $B_X$. 
 
 (3) In each fiber class $x$, the surface fiber is unique (up to isotopy), and a surface representative $F$ is a surface fiber if and only if $F$ is taut.
\end{theorem}

We call the open faces in Theorem~\ref{norm} (2) {\it fibered faces}. This theorem implies that the non-zero rational classes in the cone over the fibered faces are all the rational fiber classes.


 \begin{theorem}\label{non-decreasing}\cite{Ga}, \cite[Theorem~A.21]{La}
Suppose that $X$ is a compact orientable irreducible $3$-manifold with boundary consisting of tori, $T$ is a component of $\partial X$.  If $F$ is a properly embedded norm-minimizing surface in $X$, $\partial F\cap T=\emptyset$, then $F$ is still norm-minimizing in $H_2(X(\alpha), \partial X(\alpha);\mathbb{R})$
 for all slopes $\alpha$ on $T$ except at most one.
 \end{theorem}

 \begin{theorem}\label{Ni}\cite{Ni}. Suppose that $F$ is a compact surface and $P=F\times [0,1]$, Let  $k \subset P$ be a knot,
 $\alpha$ be a nontrivial slope on $k$. 
 If the pair $(P(k, \alpha),(\partial F)\times [0,1])$ is homeomorphic to the pair $(F\times [0,1], (\partial F) \times [0,1])$, then one can isotope $k$ so that its image in F under the natural projection
$p:F\times [0,1]\to F$
has either no crossings or exactly one crossing. Moreover, if $k$ has been isotoped so that the projection above has the minimum number of crossings and $\alpha$ is measured with respect to the ``blackboard framing'' induced by the projection, then $\alpha=1/n$ if the projection has no crossings and $\alpha=0$ if the projection has one crossing. 
\end{theorem}

We will use the following convention in this section.

\begin{convention}\label{conv:M2}
Let
$Y$ be the hyperbolic JSJ piece of $M$ containing the torus $T=\partial M$. Let $m$ be the number of open faces of the unit ball of the Thurston norm on $H_2(Y,\partial Y\setminus T; \mathbb{R})$. Let $S_0\subset A_T$ be the finite set as in Lemma~\ref{JSJ}.
\end{convention}

\begin{proposition}\label{prop:FiniteSlopes}
     There exists a subset $S_1\subset A_T$ of at most $m$ slopes, such that the set
    \[
    \{\alpha\in A_T| M(\alpha)\, \text{ is fibered}\}\setminus (\{\lambda\}\cup S_0\cup S_1)
    \]
    has the form
    \[
    S_2\cup S(R,1),
    \]
    for two subsets $S_2,R\subset A_T$ with $\frac12|S_2|+|R|\le m$. 
    
    Moreover, for every $\beta\in R$ and $\alpha\in S(\beta,1)$, we can isotope $k_{\alpha}$ so that it is contained in a fiber of $M(\alpha)$, and $\beta$ is the surface framing induced from the fiber. In this case, for any other $\alpha'\in S(\beta,1)$, $M(\alpha')$ can be obtained from $M(\alpha)$ by cutting open along the fiber then regluing via a power of the Dehn twist along $k_{\alpha}$.
\end{proposition}


\begin{proof}[Proof of Proposition \ref{hyperbolic} from Proposition~\ref{prop:FiniteSlopes}]
By Proposition~\ref{prop:FiniteSlopes}, the set $\{\alpha\in A_T| M(\alpha)\, \text{ is fibered}\}$ is the union of a finite set with
$S(R,1)$. By Remark~\ref{sparse2}, $S(R,1)$ is sparse, so $\{\alpha\in A_T| M(\alpha)\, \text{ is fibered}\}$ is also sparse by Lemma~\ref{sparse1}.
\end{proof}



  
  
  
  
  Below we are going to prove Proposition~\ref{prop:FiniteSlopes}. To do that, we need to first prove that the map $i
: H_2(Y,\partial Y\setminus T;\mathbb{R})\to H_2(Y(\alpha),\partial Y(\alpha);\mathbb{R})$ induced by inclusion  is an isometric embedding under some mild conditions (Proposition~\ref{isometry}). Then we prove that under the isometric embedding condition, the core of the Dehn filling $k_\alpha$ can be isotoped to be disjoint from the fiber 
   (Proposition~\ref{avoiding1}),  and there exist finitely many homology classes $v_i$ of $Y$ such that if the Dehn filling manifold $Y(\alpha)$ is fibered, then at least one $i(v_i)$ of $Y(\alpha)$ is a rational fiber class (Proposition~\ref{finiteness}). Those results are proved by applying Theorem~\ref{norm} and Theorem~\ref{non-decreasing}, as well as some observations
  on 3-manifold topology and convex geometry. With those results, we can reduce our problem to the the problem 
  of doing surgery on knots in a product manifold to get the same manifold, and then we can apply Theorem \ref{Ni} to finish the proof of Proposition \ref{prop:FiniteSlopes}.

Now we start the above process.
We have the inclusions $$i_T: (Y, \partial Y\setminus T)\to (Y, \partial Y), \quad i_\alpha: Y\to Y(\alpha).$$

\begin{lemma}\label{injection} 
Both
$$(i_T)_*: H_2(Y, \partial Y \setminus T; \mathbb R)\to H_2(Y, \partial Y; \mathbb R)$$  and $$(i_\alpha)_*:H_2(Y,\partial Y\setminus T; \mathbb{R})\to H_2(Y(\alpha), \partial Y(\alpha);\mathbb{R})$$ are embeddings. Moreover, the image of $(i_\alpha)_*$ has codimension $0$ or $1$.
\end{lemma}

\begin{proof} The fact that both $(i_T)_*$ and $(i_\alpha)_*$ are injections can be derived from the following geometric observation:
If  a proper oriented surface $(F, \partial F)\subset (Y, \partial Y\setminus T)$ is non-separating,
then $(F, \partial F)\subset (Y, \partial Y)$ and  $(F, \partial F)\subset (Y(\alpha), \partial Y(\alpha))$ are
both non-separating.

Now we prove the ``moreover" part.
The homology long exact sequence associated to $(Y(\alpha), Y, \partial Y\setminus T)$ is
\[
\xymatrix{
H_2(Y,\partial Y\setminus T; \mathbb{R})\ar[r]^-{(i_\alpha)_*} &H_2(Y(\alpha),\partial Y\setminus T;\mathbb{R})\ar[r] &H_2(Y(\alpha), Y, \mathbb{R})}
\]
So 
\[\mathrm{coker}(i_\alpha)_*\subset
H_2(Y(\alpha),Y)\cong H_2(D^2\times S^1, S^1\times S^1)\cong \mathbb{R}.\]
Hence \  $\dim\mathrm{coker}(i_\alpha)_*\le1$.
\end{proof}

By Lemma~\ref{injection}, the Thurston norm $th_Y$ on  $H_2(Y, \partial Y; \mathbb R)$ provides a norm $th_{Y\setminus T}$
on  $H_2(Y, \partial Y\setminus T; \mathbb R)$ via restriction. For hyperbolic $Y(\alpha)$, we also have the Thurston norm $th_{Y(\alpha)}$
on $H_2(Y(\alpha), \partial Y(\alpha); \mathbb R)$.

\begin{proposition}\label{isometry}
Assume that both $Y$ and $Y(\alpha)$ are hyperbolic. Then with the norm $th_{Y\setminus T}$ and $th_{Y(\alpha)}$ given above,
 $(i_\alpha)_*$ is an isometric embedding
(of codimension $0$ or $1$) for all slopes $\alpha\in A_T\setminus S_1$, where  $S_1\subset  A_T$ is a finite set with $|S_1|\le m$.
\end{proposition}

\begin{proof}  By Lemma~\ref{injection},  $(i_\alpha)_*$ is an embedding of codimension 0 or 1. 

Below  we often use $i_\alpha(x)$ to indicate $(i_\alpha)_*(x)$ for any $x\in H_2(Y, \partial Y\setminus T; \mathbb R)$ for brief.

Let $B_{Y\setminus T}$ be the unit ball of $H_2(Y, \partial Y \setminus T; \mathbb R)$ under the norm $th_{Y\setminus T}$,
and $B_{Y(\alpha)}$ be the unit ball of $H_2(Y(\alpha), \partial Y(\alpha); \mathbb R)$ under the norm $th_{Y(\alpha)}$.
First note that 
$$th_{Y(\alpha)}(i_\alpha(x))\le  th_{Y\setminus T}(x)$$
for any $x\in H_2(Y, \partial Y\setminus T; \mathbb R)$, since Thurston norm is not increasing under maps.

Let $P_i$ be the open faces of $B_{Y\setminus T}$, $v_i$ be the barycenter of $P_i$, $i=1,..., m$.
 Since the Thurston norm $th_{Y\setminus T}$ on $P_i$ is $1$, the Thurston norm $th_{Y(\alpha)}$ satisfies
 \begin{equation}
 \label{eq1}
 th_{Y(\alpha)}(i_\alpha(x))\le1,\quad\text{for every }x\in P_i.
\end{equation}
 
 By Theorem \ref{norm} (1), we conclude that all $v_i$ are rational.
For each $v_i$, let $F_i$ be a compact surface such that $[F_i]$ is a multiple of $v_i$ and $F_i$ is norm-minimizing.
Each $F_i$ is still norm-minimizing in $Y(\alpha)$ for all slopes in $T$ except at most one by Theorem~\ref{non-decreasing}.
That is to say $th_{Y(\alpha)}(i_\alpha[F_i])=th_{Y\setminus T}([F_i])$, so 
$$th_{Y(\alpha)}(i_\alpha(v_i))=th_{Y\setminus T}(v_i)=1.$$

It concludes that for all slopes $\alpha\in A_T$ with  at most $m$ exceptions, the Thurston norm $th_{Y(\alpha)}$ on every $i_\alpha(v_i)$  is $1$. By (\ref{eq1}), we conclude that the restriction of $th_{Y(\alpha)}$ on the convex set $i_\alpha(P_i)$ reaches its maximum in its interior point $i_\alpha(x_i)$ (since $i_\alpha$ is an embedding).
By the convexity of the Thurston norm $th_{Y(\alpha)}$, we get $th_{Y(\alpha)}\equiv1$ on $i_\alpha(P_i)$.
So  we conclude that $th_{Y(\alpha)}\equiv1$ on $i_\alpha(\partial B_{Y\setminus T})$. Then $i_\alpha$ is an isometric embedding.
\end{proof}

\begin{proposition}\label{avoiding1} Suppose that $Y$ is hyperbolic. 
  If $M(\alpha)$ is fibered by $F$,
  then $F$ is disjoint from $k_\alpha$ for any $\alpha\in A_T\setminus (\{\lambda\}\cup S_0\cup S_1)$, where $S_1$ is the subset of $A_T$ with $|S_1|\le m$ as in Proposition \ref{isometry}.
  \end{proposition}
\begin{proof}

Since $\alpha\notin S_0$, $Y(\alpha)$ is a hyperbolic piece of $M(\alpha)$ by our choice of $S_0$.
If $M(\alpha)$   is fibered by a surface $F$ then $Y(\alpha)$ is fibered by $F_0$, a component of $F\cap Y$, by Lemma~\ref{topology1}. 
Since $\alpha\ne\lambda$, by Propositon~\ref{zero},  $[F]\cdot[k_\alpha]=0$. Since $$F\cap k_\alpha=(F\cap Y(\alpha))\cap k_\alpha,$$ 
it follows that $[F_0]\cdot[k_\alpha]=0$ in $Y(\alpha)$. 

As in Proposition~\ref{isometry}, there is a finite set $S_1\subset A_T$ with $|S_1|\le m$, such that $i_\alpha$ is an isometric embedding with respect to the Thurston norm for any $\alpha\notin S_1$. 
Now our conclusion follows from Lemma~\ref{avoid3}.
\end{proof}

\begin{lemma}\label{avoid3}Suppose that $Y$ is hyperbolic, and
$Y(\alpha)$ is fibered by $F_0$ with $[F_0]\cdot[k_\alpha]=0$. If $i_\alpha$ is an isometric embedding with respect to the Thurston norm, then $F_0$ can be isotoped to be disjoint from $k_\alpha$.
\end{lemma}



\begin{proof} Since $[F_0]\cdot[k_\alpha]=0$, we can surger $F_0$ to $F_0'$ in a tubular neighborhood of $k_\alpha$ such that $F_0$ and $F_0'$ are homologous in $Y(\alpha)$ relative to $\partial Y(\alpha)$,  and $F_0'\cap k_\alpha=\emptyset.$  Now we have $F_0'\subset Y$. 

We  choose a taut surface $F_{\min}'$ in the relative homology class $[F_0']\in H_2(Y,\partial Y\setminus T; \mathbb{R})$. In particular,
\[
\chi_-(F_{\min}')=th_{Y\setminus T}([F_{\min}']).
\]

Since $F_0$ is a surface fiber, by Theorem~\ref{norm}, $F_0$ is taut. So
\[
\chi_-(F_0)=th_{Y(\alpha)}([F_0]).
\]
Since
$$i_\alpha([F_{\min}'])=i_\alpha([F'])=[F_0]$$ 
and $i_\alpha$ is norm-preserving, we have that 
$$th_{Y\setminus T}([F_{\min}'])=th_{Y(\alpha)}([F_0]),$$ so 
$$\chi_-(F_{\min}')=\chi_-(F_0).$$

Therefore, $F_{\min}'$ is also a taut surface in $Y(\alpha)$. However, in the fiber class $[F_0]\in H_2(Y(\alpha),\partial Y(\alpha);\mathbb{R})$, the surface fiber $F_0$ is the unique taut surface up to isotopy. Hence $F_0$ is isotopic to $F_{\min}'$, which is disjoint from $k_\alpha$.
\end{proof}


\begin{lemma}\label{iff} Assume that both $Y$ and $Y(\alpha)$ are hyperbolic.
Suppose that
$(i_\alpha)_*$ is an isometric embedding.
 Then for any pair of rational homology classes $(x_1,x_2)$ lying in the same open face of $B_{Y\setminus T}$, $i_\alpha(x_1)$ is a rational fiber class if and only if $i_\alpha(x_2)$ is.
\end{lemma}

\begin{proof} 

Since $x_1, x_2$ lie in the same open face of $B_{Y\setminus T}$, there is a line segment $L\subset \partial B_{Y\setminus T}$ whose interior contains $x_1$ and $x_2$, and $th_{Y\setminus T}$ is $1$ on $L$. 

Since $(i_\alpha)_*$ is an isometric embedding, $th_{Y(\alpha)}$ is $1$ on $i_\alpha(L)$. So $i_\alpha(L)$ lies in $\partial B_{Y(\alpha)}$. Let $P\subset B_{Y(\alpha)}$ be the minimal dimensional cell containing $i_\alpha(L)$. 
Since $(i_\alpha)_*$ is linear, $i_\alpha(x_1)$ and $i_\alpha(x_2)$ are in the interior of $P$. If $P$ is of dimension $n-1$, then both $i_\alpha(x_1)$ and $i_\alpha(x_2)$
lie in the same open face, and by Theorem~\ref{norm}, $i_\alpha(x_1)$ is a rational fiber class if and only if $i_\alpha(x_2)$ is. If $P$ is of dimension $< n-1$, then none of $i_\alpha(x_1)$ and $i_\alpha(x_2)$ is a rational fiber class,
still by Theorem~\ref{norm}.
\end{proof}

\begin{proposition}\label{finiteness} Assume that both $Y$ and $Y(\alpha)$ are hyperbolic. Suppose also that
$(i_\alpha)_*$ is an isometric embedding. For each open face $P_i$ of 
$B_{Y\setminus T}$, let $v_i$ be its barycenter, $i=1,\dots,m$. 

 Then
 $Y(\alpha)$ has a surface fiber disjoint from the core $k_\alpha$ of the Dehn filling if and only if one of the $i_\alpha(v_i)$'s is a rational fiber class in $H_2(Y(\alpha),\partial Y(\alpha); \mathbb{R})$.
\end{proposition}

\begin{proof}The ``only if" part:
Suppose that $Y(\alpha)$ has a surface fiber disjoint from the core $k_\alpha$ of the Dehn filling, then there is a class $y\in H_2(Y,\partial Y\setminus T; \mathbb{R})$ such that $i_\alpha(y)$ is a fiber class in $H_2(Y(\alpha))$. Let $x=y/th_{Y\setminus T}(y)$, then $th_{Y\setminus T}(x)=1$.

{\bf Claim.} $x$ is in an open face of $B_{Y\setminus T}$.

Assuming the claim, by Lemma~\ref{iff} we get $i_\alpha(v_i)$ is also a rational fiber class, where $v_i$ is the barrycenter of the open face containing $x$. This finishes the proof of the ``only if" part of Proposition~\ref{finiteness}.

The remaining is to prove the Claim. Since $i_\alpha(x)$ is a rational fiber class, $i_\alpha(x)$ lies in an open face $P\subset B_{Y(\alpha)}$ by Theorem~\ref{norm}. Let $H=i_\alpha(H_2(Y, \partial Y\setminus T); \mathbb{R})$ and $$P'=P\cap H.$$
Notice that either  $H=H_2(Y(\alpha),\partial Y(\alpha);\mathbb{R})$, or $H$ is a hyperplane in $H_2(Y(\alpha),\partial Y(\alpha); \mathbb{R})$
by Proposition~\ref{isometry}.

If $H=H_2(Y(\alpha),\partial Y(\alpha); \mathbb{R})$, then $i_\alpha$ is an isometry, and the Claim is obvious. 

If $H$ is a hyperplane in $H_2(Y(\alpha),\partial Y(\alpha); \mathbb{R})$,
then $$H\cap B_{Y(\alpha)}=i_\alpha(B_{Y\setminus T}),  \text{ and } P'\subset \partial (i_\alpha(B_{Y\setminus T})).$$

Since the supporting hyperplane of the face $P$ does not contain the origin, and $H$ contains the origin and the interior point $i_\alpha(x)$
of $P$, for $P'=P\cap H$, we have 
$$\dim P'=\dim P-1,$$
that is, $P'$ is an open face in $i_\alpha (B_{Y\setminus T})$. So $i_\alpha^{-1}(P')$ is an open face in $B_{Y\setminus T}$ containing $x$. This completes the proof of the claim.

The ``if" part: 
Suppose that there exists $i\in\{1,...,m\}$ such that $i_\alpha(v_i)$ is a rational fiber class in $H_2(Y(\alpha),\partial Y(\alpha);\mathbb{R})$. Take $r\in\mathbb{Q}\setminus \{0\}$ such that $rv_i$ is a primitive integer homology class in $H_2(Y,\partial Y\setminus T;\mathbb{R})$. The following argument is similar  to the proof of Lemma \ref{avoid3}. We choose a taut surface $F_{\rm min}\subset Y$  representing $rv_i$. Then $F_{\rm min}\cap T=\emptyset$ and
$$\chi_-(F_{\rm min})=th_{Y\setminus T}(rv_i).$$
Since the inclusion $i_\alpha: Y\to Y(\alpha)$ induces an isometric embedding
$$(i_\alpha)_*: H_2(Y,\partial Y\setminus T;\mathbb{R})\to H_2(Y(\alpha),\partial Y(\alpha);\mathbb{R})$$
with respect to the Thurston norm, we have $$th_{Y(\alpha)}(i_\alpha(rv_i))=th_{Y\setminus T}(rv_i).$$ Thus $$\chi_-(F_{\rm min})=th_{Y(\alpha)}(i_\alpha(rv_i)).$$ Then $F_{\rm min}$ is also a taut surface in $Y(\alpha)$. Note that $i_\alpha(rv_i)$ is a fiber class. By Theorem \ref{norm}, $F_{\rm min}$ is a surface fiber in $Y(\alpha)$. Note that $F_{\rm min}\cap T=\emptyset$, thus $F_{\rm min}$ is disjoint from the core $k_\alpha$ of the Dehn filling. This finishes the proof of the ``if" part of Proposition \ref{finiteness}.
\end{proof}


\begin{proof}[Proof of Proposition~\ref{prop:FiniteSlopes}] We only consider the slopes $\alpha$ in $A_T\setminus (\{\lambda\}\cup S_0\cup S_1)$. Recall that the 
$S_1$ in Proposition~\ref{avoiding1} is the one given in Proposition~\ref{isometry}. For each such slope $\alpha$, by Proposition~\ref{isometry},
$i_\alpha$ is an isometric embedding.
By Proposition~\ref{finiteness}, 
there are finitely many classes $$v_1,...,v_m\in H_2(Y,\partial Y\setminus T;\mathbb R)$$ such that
 $Y(\alpha)$ has a surface fiber disjoint from the core $k_\alpha$ of the Dehn filling if and only if one of $i_\alpha(v_1),...,i_\alpha(v_m)$ is a rational fiber class in $H_2(Y(\alpha),\partial Y(\alpha);\mathbb R)$.
For each $i$, $1\le i \le m$, let 
$$S'_i=\{\alpha\in A_T\setminus (\{\lambda\}\cup S_0\cup S_1)| \text{$i_\alpha(v_i)$ is a rational fiber class in $H_2(Y(\alpha),\partial Y(\alpha);\mathbb R)$}\}.$$

{\bf Claim.} For each $i$, either $|S'_i|\le2$, or $S'_i=S(\beta_i,1)$ for a slope $\beta_i$. In the latter case, for each $\alpha\in S'_i$, we can isotope $k_{\alpha}$ so that it is contained in a fiber of $M(\alpha)$, and $\beta_i$ is the surface framing induced from the fiber.

Assume the claim, let $S_2$ be the union of those $S'_i$ with $|S'_i|\le2$, and let $R$ consist of all $\beta_i$'s with $S'_i=S(\beta_i,1)$. Then $\frac12|S_2|+|R|\le m$, and for any $\alpha\notin S_2\cup S(R,1)$,  $Y(\alpha)$ has no surface fiber disjoint from $k_\alpha$ by Proposition~\ref{finiteness}, since none of $i_\alpha(v_i)$ is a rational fiber class. When $\alpha\in S(\beta_i,1)$, by the Claim $k_{\alpha}$ can be isotoped to be contained in a fiber of $M(\alpha)$. It is a standard fact that, for any other $\alpha'\in S(\beta_i,1)$, $M(\alpha')$ can be obtained from $M(\alpha)$ by cutting open along the fiber then regluing via a power of the Dehn twist along $k_{\alpha}$.
We then get the conclusion of Proposition~\ref{prop:FiniteSlopes}.

The remaining is to prove the Claim. Without loss of generality, assume $S'_i\ne \emptyset$. Given $\alpha\in S'_i$, $Y(\alpha)$ has a connected surface fiber $F_\alpha$ realizing a multiple of the rational fiber class $i_\alpha(v_i)\in H_2(Y(\alpha),\partial Y(\alpha);\mathbb{R})$ and avoiding $k_\alpha$.

Now we argue $F_\alpha$ is independent of $\alpha\in S'_i$ up to isotopy.  
Recall $Y(\alpha)=Y\cup N(k_\alpha)$, and $F_\alpha\cap k_\alpha=\emptyset$ is equivalent to (after an isotopy) $F_\alpha\subset Y$.
Let $F_i\subset Y$ be a taut surface representing a multiple of the class $v_i$. Then $F_i$ is also a taut surface in $Y(\alpha)$ representing a multiple of $i_\alpha(v_i)$. 
Since $F_\alpha$ is a fiber surface, $F_i$ consists of several parallel copies of $F_\alpha$ up to isotopy. We can replace $F_i$ with a component of it if necessary. 

Now for every $\alpha\in S_i'$, we have 
$$F:=F_i=F_\alpha\subset Y \subset Y(\alpha)$$ 
and 
$$F\times [0,1]=Y(\alpha)\bbslash F=(Y\bbslash F)\cup N(k_\alpha),$$
where we consider $k_\alpha$ as a knot in $F\times [0,1]$.
Fix an $\alpha_0\in S_i'$. If $|S_i'|>1$,
suppose $\alpha\ne \alpha_0$ is any other element in $S_i'$. Then $Y(\alpha)$ is a non-trivial surgery on $Y(\alpha_0)$ along $k_{\alpha_0}$. Then $Y(\alpha)\bbslash F$ is also a non-trivial surgery on $Y(\alpha_0)\bbslash F$ along $k_{\alpha_0}$, and we have 
 \[(Y(\alpha)\bbslash F,\partial Y(\alpha)\bbslash\partial F)\cong(F\times [0,1],\partial F\times[0,1])\cong(Y(\alpha_0)\bbslash F,\partial Y(\alpha_0)\bbslash\partial F).\] 
 By Theorem~\ref{Ni}, one can isotope $k_{\alpha_0}$ such that its image of $k_{\alpha_0}$ in $F$ under the natural projection
 $p: Y(\alpha_0)\bbslash F\cong F\times [0,1]\rightarrow F$ has at most one crossing. Moreover,  if $k_\alpha$ has been isotoped so that the projection above has the minimum number of crossings, then 
 \begin{enumerate}
 \item
If the projection above
  has one crossing, then the surgery slope on $k_{\alpha_0}$ to obtain $Y(\alpha)\bbslash F$  is $0$ with respect to the ``blackboard framing".
 \item If  the projection above has no crossing, then the surgery slope on $k_{\alpha_0}$ to obtain $Y(\alpha)\bbslash F$  is $1/n$ with respect to the ``blackboard framing".
 \end{enumerate}
 In the former case, $|S_i'|=2$. In the latter case, $k_{\alpha_0}$ can be isotoped to be a simple closed curve in $F$, and
 $S_i'=S(\beta_i,1)$, where $\beta_i$ is the surface framing. The claim is proved.
 
This completes the proof of Proposition \ref{prop:FiniteSlopes}.
\end{proof}

\section{The case of Seifert boundary piece}

We state Theorem \ref{main1} when $Y$ is a Seifert piece as follows.

\begin{proposition} \label{Seifert} Suppose that $Y$ is Seifert fibered. Then the set $\{\alpha\in A_T| M(\alpha) \text{ is fibered}\}$ is sparse.
\end{proposition}

\begin{lemma}\label{topology2}
Suppose that $Y$ is a Seifert fibered $3$-manifold over a base orbifold with negative orbifold Euler characteristic. Let $f_0$ be a Seifert fiber of $Y$. 
If $Y$ is fibered by a surface $F$, then $[F]\cdot [f_0] \ne 0$. 
\end{lemma}
\begin{proof} Suppose that $Y=\mathcal M(F, \phi)$, where $\phi$ is a periodic map on $F$. Then $\phi^n=\mathrm{id}$ for some integer  $n$.
So we have the cyclic covering \[p: F\times S^1=\mathcal M(F, \mathrm{id})\to Y=\mathcal M(F, \phi),\] which preserves both $F$-fibers and $S^1$-fibers.
Orient $Y$, and $F, f_0$ in $Y$, and lift those orientations to their pre-images under $p$.
We have $$[p^{-1}(F)]\cdot [p^{-1}(f_0)]=n [F]\cdot [f_0].$$
It is clear the left hand side is non-zero. So $[F]\cdot [f_0]\ne 0$.
\end{proof}

\begin{corollary}\label{Seifert1} Suppose that $M\ne N(K)$.
If $M$ contains a Seifert fibered piece $M_i$ whose $S^1$ fiber $f_0$ represents zero in $H_1(M; \mathbb R)$, then $M$ is not fibered.
\end{corollary}
\begin{proof} 
Suppose that $M$ is fibered by a surface $F$,  then $M_i$ is fibered by a surface $F_i$ by Lemma~\ref{topology1}, where $F_i$ is a component of $F\cap M_i$. Since $M\ne N(K)$, $M_i\ne N(K)$ is Seifert fibered over a base orbifold with negative orbifold Euler characteristic. By 
Lemma~\ref{topology2}, $[F_i]\cdot [f_0] \ne 0$, hence  $[F]\cdot [f_0] \ne 0$, which implies that   $[f_0]\ne 0$ in $H_1(M; \mathbb R)$. This is a contradiction.
\end{proof}

\begin{proof}[Proof of Proposition \ref{Seifert}]
The proof is divided into two cases:

(1) $M\ne N(K)$:
Let $f_0\subset T=\partial M\subset Y$ be the $S^1$-fiber of $Y$. 

 By Lemma~\ref{JSJ}, there is a sparse set $S_0\subset A_T$ such that for each $\alpha \in A_T\setminus S_0$,
  $M(\alpha)$ is still irreducible, and $Y(\alpha)$ is still a Seifert fibered piece with $S^1$ fiber $f_0$.
  Below we assume that  $\alpha \in A_T\setminus S_0$. 

For any slope $\alpha\ne\lambda$ on $T$,
$[f_0]=0$ in $H_1(M(\alpha);  \mathbb R)$ by Proposition~\ref{zero}.
  Then by Corollary~\ref{Seifert1}, $M(\alpha)$ is not fibered.

 (2) $M=N(K)$: In this case, $b_1(N(K))=1$. So 
 $b_1(N(K)(\alpha))=0$ for $\alpha\ne\lambda$. Hence $N(K)(\alpha)$ can not be a surface bundle over $S^1$ for $\alpha\ne\lambda$.

  We have finished the proof of Proposition~\ref{Seifert}.
\end{proof}
 \begin{proof}[Proof of Theorem \ref{main1}]
 It follows from Proposition \ref{hyperbolic} and Proposition \ref{Seifert}.
 \end{proof}
\section{Proof of the corollary}

\begin{proof}[Proof of Corollary \ref{surgery}]
Let $X$ be a closed orientable 3-manifold  and $k$ be a knot in $X$. Let  $M=X\setminus N(k).$
Since $\partial M$ is a torus, either $M$ is irreducible or $M$ is a nontrivial connected sum.

Suppose first that $M$ is irreducible. Then either $M$ is a solid torus or $M$ has incompressible boundary. In the former case, $X(k,\alpha)$ is covered by $S^3$ for all but one $\alpha$, and the conclusion follows. In the latter case, the conclusion follows from Theorem~\ref{main1}.

If $M$ is not irreducible, then $M=X_1\# M'$, where $X_1\ne S^3$ is a closed orientable 3-manifold and $M'$ is a compact 3-manifold
with torus boundary, and 
$$X(k, \alpha)=M(\alpha)=X_1\# M'(\alpha).$$
If $X(k, \alpha)$ is fibered, then $X(k, \alpha)$ is prime, so $M'(\alpha)$ is $S^3$, that is $M'=S^3\setminus N(k')$
for a knot $k'$ in $S^3$. If $k'$ is knotted, then  $\alpha$ is unique by \cite{GL}. If $k'$ is unknotted, then $M'$ is the solid torus,
and $\alpha=1/n$, $n\in \mathbb Z$. Then $\alpha\in S(\beta,1)$ where $\beta$ is the meridian curve of the solid torus $M'$. By Remark \ref{sparse2}, $S(\beta,1)$ is a sparse set. 

So $X(k, \alpha)$ is non-fibered for all but a sparse set of $\alpha$.  This completes the proof of the corollary.
\end{proof}

\bibliographystyle{amsalpha}

\begin{thebibliography}{GWWZ}




\bibitem[FM]{FM} B. Farb, D. Margalit, {\it A primer on mapping class groups.} Princeton Mathematical Series, 49. Princeton University Press, Princeton, NJ, 2012.


\bibitem[Ga]{Ga} D. Gabai, {\it Foliations and the topology of 3-manifolds, II,}  J. Differ. Geom. 26 (1987), 461-478.


\bibitem[GL]{GL} C. Gordon and J. Luecke, {\it Knots are determined by their complements}, J. Amer. Math. Soc. 2 (1989)
371-415.


\bibitem[Ha]{Ha} A. Hatcher, {\it Notes on basic $3$-manifold topology}. \url{http://pi.math.cornell.edu/~hatcher/.}




\bibitem[La]{La} M. Lackenby,  {\it Dehn surgery on knots in 3-manifolds.}
J. Amer. Math. Soc. 10 (1997), no. 4, 835-864.

\bibitem[LMST]{LMST}
C. Leininger, Y. N. Minsky, J. Souto, S. J. Taylor, {\it
Weil-Petersson translation length and manifolds with many fibered fillings.}
Adv. Math. 376 (2021), Paper No. 107457, 54 pp.



\bibitem[Ni]{Ni} Y. Ni, 
{\it Dehn surgeries on knots in product manifolds.}
J. Topol. 4 (2011), no. 4, 799-816.



\bibitem[Th1]{Th1}
W. Thurston,
{\it Three dimensional manifolds, Kleinian groups and hyperbolic geometry,}  Bull. Amer. Math. Soc. 6 (1982) 357-381.



\bibitem[Th2]{Th2} W. Thurston, 
{\it A norm for the homology of 3-manifolds.}
Mem. Amer. Math. Soc. 59 (1986), no. 339, i-vi and 99-130.





\end{thebibliography}

\end{document}